\documentclass[12pt]{article}
\PassOptionsToPackage{dvipsnames,svgnames,x11names}{xcolor}
\usepackage{xcolor}
\definecolor{labelkey}{rgb}{0,0.08,0.45}
\definecolor{refkey}{rgb}{0,0.6,0.0}
\definecolor{Brown}{rgb}{0.45,0.0,0.05}
\definecolor{lime}{rgb}{0.00,0.8,0.0}
\definecolor{lblue}{rgb}{0.5,0.5,0.99}
\definecolor{OliveGreen}{rgb}{0,0.6,0}
\definecolor{tyrianpurple}{rgb}{0.4, 0.01, 0.24}
\definecolor{myseagreen}{HTML}{3FBC9D}
\definecolor{myblue}{rgb}{0.9,0.9,0.98}
\colorlet{hlcyan}{cyan!30}

\usepackage{mathpazo} % With old-style figures and real smallcaps
\usepackage{subcaption} % Subfigure
\usepackage[T1]{fontenc}
\usepackage{mathtools}
\usepackage{amssymb}
\usepackage{stmaryrd}
\usepackage{empheq}

\usepackage{hyperref}
\hypersetup{
  colorlinks=true,
  linkcolor=refkey,
  urlcolor=lblue,
  citecolor=magenta
}

\usepackage{amsthm}
\makeatletter
\def\th@plain{%
  \thm@notefont{}%
  \itshape % body font
}
\def\th@definition{%
  \thm@notefont{}%
  \normalfont % body font
}

\usepackage[capitalize,nameinlink]{cleveref}
\makeatother
\newtheorem{theorem}{Theorem}[section]
\newtheorem{lemma}[theorem]{Lemma}

\newtheorem{proposition}[theorem]{Proposition}

\newtheorem{example}[theorem]{Example}
\newtheorem{fact}[theorem]{Fact}
\newtheorem{remark}[theorem]{Remark}
\crefname{theorem}{Theorem}{Theorems}
\Crefname{theorem}{Theorem}{Theorems}
\crefname{fact}{Fact}{facts}
\Crefname{fact}{Fact}{facts}
\crefname{equation}{}{equations}
\crefname{chapter}{Appendix}{chapters}
\crefname{item}{}{items}
\crefname{enumi}{}{}

\usepackage{nicematrix}
\usepackage{booktabs}
\usepackage{siunitx}
\usepackage{tikz}
\usepackage[color=myseagreen]{todonotes}
\usepackage{soul}
\usepackage{nameref}
\usepackage{graphicx}
\usepackage{float}
\usepackage[shortlabels,inline]{enumitem}
\setlist[enumerate]{nosep}
\usepackage{relsize}

\usepackage[margin=1in,footskip=0.4in]{geometry}
\makeatletter

\let\orig@label\label

\renewcommand{\label}[1]{%
  \begingroup
  \def\@currentlabelname{}%
  \ifx\current@theorem\relax\else
    \def\@currentlabelname{\current@theorem}%
  \fi
  \ifx\cref@currentlabel\undefined\else
    \let\@currentlabelname\cref@currentlabel
  \fi
  \orig@label{#1}%
  \endgroup
}

\AtBeginEnvironment{theorem}{\def\current@theorem{theorem}}
\AtBeginEnvironment{proposition}{\def\current@theorem{proposition}}
\AtBeginEnvironment{fact}{\def\current@theorem{fact}}
\makeatother

\newcommand{\seppthree}{\setlength{\itemsep}{-3pt}}

\newcommand{\bx}{\ensuremath{\mathbf{x}}}
\newcommand{\bu}{\ensuremath{\mathbf{u}}}

\newcommand{\bzero}{\ensuremath{\boldsymbol{0}}}

\newcommand{\weakly}{\ensuremath{\:{\rightharpoonup}\:}}

\newcommand{\kkk}{\ensuremath{k\in\mathbb{N}}}
\newcommand{\thalb}{\ensuremath{\tfrac{1}{2}}}
\newcommand{\menge}[2]{\big\{{#1}~\big|~{#2}\big\}}

\newcommand{\To}{\ensuremath{\rightrightarrows}}

\newcommand{\scal}[2]{\left\langle{#1},{#2}\right\rangle}

\newcommand{\RR}{\ensuremath{\mathbb{R}}}

\newcommand{\RP}{\ensuremath{\mathbb{R}_+}}

\newcommand{\NN}{\ensuremath{\mathbb{N}}}

\newcommand{\ran}{\ensuremath{\operatorname{ran}\,}}

\newcommand{\zer}{\ensuremath{\operatorname{zer}}}

\newcommand{\conv}{\ensuremath{\operatorname{conv}\,}}

\newcommand{\cconv}{\ensuremath{\overline{\operatorname{conv}}\,}}

\newcommand{\cspan}{\ensuremath{\overline{\operatorname{span}}\,}}

\newcommand{\Fix}{\ensuremath{\operatorname{Fix}}}
\newcommand{\Id}{\ensuremath{\operatorname{Id}}}

\newcommand{\bX}{\ensuremath{\mathbf{X}}}
\newcommand{\bY}{\ensuremath{\mathbf{Y}}}

\newcommand{\bT}{\ensuremath{\mathbf{T}}}
\newcommand{\bC}{\ensuremath{\mathbf{C}}}
\newcommand{\bN}{\ensuremath{\mathbf{N}}}

\newcommand{\bA}{\ensuremath{\mathbf{A}}}
\newcommand{\bK}{\ensuremath{\mathbf{K}}}

\newcommand{\by}{\ensuremath{\mathbf{y}}}
\newcommand{\bz}{\ensuremath{\mathbf{z}}}

\providecommand{\norm}[1]{\lVert#1\rVert}

\providecommand{\ccone}{\overline{\operatorname{cone}}}

\providecommand{\gra}{\operatorname{gra}}

\newcommand{\cran}{\ensuremath{\overline{\operatorname{ran}}\,}}

\newcommand{\mybluebox}[1]{\colorbox{myblue}{\hspace{1em}#1\hspace{1em}}}

\author{
Heinz H.\ Bauschke\thanks{
Mathematics, University
of British Columbia,
Kelowna, B.C.\ V1V~1V7, Canada. E-mail:
\texttt{heinz.bauschke@ubc.ca}.}~~~and~
Tran Thanh Tung\thanks{
Mathematics, University
of British Columbia,
Vancouver, B.C.\ V6T~1Z4, Canada. 
 E-mail: \texttt{tung.tran@ubc.ca}.}
}

\title{\textsf{
    Ces\`aro means of firmly nonexpansive iterates need not 
    converge strongly}
}

\date{October 29, 2026}

\begin{document}
\allowdisplaybreaks
\maketitle

\begin{abstract}
Firmly nonexpansive operators arise naturally as resolvents of monotone operators and as generalizations of projections and proximal mappings in convex optimization and fixed point theory. 
While their iterates are known to converge weakly to a fixed point, strong convergence is not guaranteed (Genel and Lindenstrauss, 1975).
Strong convergence of Ces\`aro means of iterates 
is also known to fail for general nonlinear nonexpansive mappings (Krengel and Lin, 1987).

In this paper, we show that this failure persists in the much smaller 
class of firmly nonexpansive mappings. 
Using suitable meshes, we construct a new explicit family of
counterexamples in infinite-dimensional Hilbert spaces with the origin 
as the unique fixed point.
In the harmonic case, the Ces\`aro
means of the iterates remain bounded away from the origin. 
Another variant yields Ces\`aro means that converge strongly to the origin.
A third variant presents Ces\`aro means whose norms oscillate in the sense that their
liminf is zero while their limsup is positive. 
Thus the strong-convergence conclusion in von Neumann’s 
linear mean ergodic theorem does not extend to 
Baillon’s nonlinear mean ergodic theorem, even 
for firmly nonexpansive mappings.
\end{abstract}



\section{Introduction}

\label{s:intro}

Throughout this paper, 
\begin{empheq}[box=\mybluebox]{equation}
\label{e:X}
  \text{$\bX$ is an infinite-dimensional real Hilbert space, 
  }
\end{empheq}
  with inner product $\scal{\cdot}{\cdot}$ and induced norm $\norm{\cdot}$. 
  
Suppose that 
\begin{equation}
\bT\colon \bX\to \bX 
\;\;\text{is nonexpansive, with}
\;
\Fix \bT\neq \varnothing.
\end{equation}
Finding a fixed point of $\bT$ is a central problem in optimization and variational analysis. Unfortunately, without any additional assumptions, 
iterating $\bT$ may not yield a solution (consider $\bT=-\Id$ with a starting point that is not the origin.) 
If, however, $\bT$ is $\alpha$-averaged nonexpansive, i.e., 
$\bT$ can be written 
as $(1-\alpha)\Id+\alpha \bN$, with $\alpha \in\left[0,1\right[$ and 
$\bN\colon \bX\to \bX$ nonexpansive, then the iterates of $\bT$ 
converge weakly to
a point in $\Fix \bT$. An important special case is when $\bT$ is 
firmly nonexpansive, i.e., $\thalb$-averaged nonexpansive. 
An obvious question is whether the convergence can fail to be strong.
The answer is affirmative due to a now-classical example from 1975:

\begin{example}[Genel-Lindenstrauss] {\rm (See \cite{GL}.)}
\label{f:GL}
Suppose that $\bX=\ell^2$. Then there exist 
a bounded closed convex subset $\bC$ of $\bX$, 
a firmly nonexpansive mapping $\bT\colon \bC\to \bC$, and a starting point 
$\bx\in \bC$ such that $\bT^n {\bx}\weakly \bzero\in \Fix \bT$ but 
$\inf_{n\geq 1}\|\bT^n\bx\|\geq \thalb$.
\end{example}

Hence iterating $\bT$, even when $\bT$ is firmly nonexpansive, need not necessarily produce a sequence that converges strongly 
to a fixed point of $\bT$. 
It is thus tempting to consider Ces\`aro means of the iterates.
Indeed, von Neumann's linear mean ergodic theorem gives strong convergence
in the linear case even when $\bT$ is merely assumed to be nonexpansive:

\begin{fact}[von Neumann] 
\label{f:vN}
{\rm (See \cite{vN}, and also \cite{Baillon76}, \cite{RieszNagy43}, \cite[Section~144]{RieszNagy}, and \cite{Wittmann}.)} 
Suppose that $\bT\colon \bX\to\bX$ is linear and nonexpansive. 
Then 
\begin{equation}
\frac{1}{n}\sum_{k=1}^n \bT^{k-1}\bx \to P_{\Fix \bT}\bx. 
\end{equation}
\end{fact}

The nonlinear variant of \cref{f:vN} was provided in 1975 by Baillon:

\begin{fact}[Baillon] 
\label{f:Baillon}
{\rm (See \cite{Baillon75}, and also \cite{Reich78}.)} 
Suppose that $\bT\colon\bX\to\bX$ is nonexpansive, 
with $\Fix \bT\neq\varnothing$. 
Then 
\begin{equation}
\frac{1}{n}\sum_{k=1}^n \bT^{k-1}\bx \weakly 
\text{ some point in $\Fix \bT$.}
\end{equation}
\end{fact}

For general nonlinear nonexpansive mappings, strong convergence in Baillon’s theorem
is known to fail; see, for example, 
Krengel and Lin’s work \cite{KrengelLin} and Benyamini and Lindenstrauss's \cite[example on page~74]{BL}. 
What appears not to have been settled is whether such failure can occur 
for \emph{firmly} nonexpansive mappings.
\emph{The goal of this work is to settle this question.} 
Our main results can be summarized as follows:

\begin{itemize}[itemsep=00pt]
\item[\bf R1]  We will provide a new family of \emph{firmly} nonexpansive 
mappings $\bT$ such that $(\bT^n\bx)_{n\geq 1}$ converges weakly 
but not strongly 
(see \cref{t:main1}). 
This is similar in spirit to \cref{f:GL}; however, our construction is simpler. 
\item[\bf R2] For an incarnation of $\bT$ from \textbf{R1}, we will show that
even the Ces\`aro means of the iterates fail to converge strongly and 
actually stay
a positive distance away from the unique fixed point of $\bT$ 
(see \cref{ex:harmCes}). 
\item[\bf R3] A second incarnation of $\bT$ from \textbf{R1} yields 
well-behaved Ces\`aro means of the iterates that do converge strongly.
(see \cref{ex:Cnice}). 
\item[\bf R4] We will also present a third incarnation where one 
subsequence of the Ces\`aro means stays away from the unique fixed point of 
$\bT$ while another subsequence converges strongly to the fixed point 
(see \cref{ss:sumbiz}). 
\end{itemize}

The remainder of this paper is organized as follows. 
In \cref{s:motiv}, we provide the motivation for our 
abstract construction while realizations are discussed in 
\cref{s:constr}. 
In \cref{s:aux}, we record a few inequalities that will be useful
in later sections. 
We then introduce the underlying curve and mesh in \cref{s:curvmesh}.
These will play a central role in \cref{s:seq}, where 
we introduce the sequence of vectors that will eventually be the iterates.
\cref{s:main} contains the main results 
\textbf{R1}, \textbf{R2}, \textbf{R3} while \textbf{R4} is presented in \cref{s:bizarre}.

The notation we employ is fairly standard and follows largely \cite{BC2017}. 

\section{Motivation for the construction and the Gaussian kernel}

\label{s:motiv}

Suppose that $\bT\colon \bX\to \bX$ is a firmly nonexpansive 
mapping such that $\Fix \bT = \{\bzero\}$, with 
$\bx_{n} = \bT^{n-1}\bx_1\weakly \bzero$ but $\bx_n\not\to \bzero$.
Because $\bT$ is firmly nonexpansive, we have 
$\|\bT\bx_n -\bT\bzero\|^2 \leq \scal{\bx_n-\bzero}{\bT\bx_n - \bT\bzero}$, i.e., 
\begin{equation}
\label{e:ant1}
\|\bx_{n+1}\|^2 \leq \scal{\bx_n}{\bx_{n+1}}.
\end{equation}
Cauchy--Schwarz now implies that $\|\bx_{n+1}\|\leq\|\bx_n\|$; hence, 
$(\|\bx_n\|)_{n\geq 1}$ is decreasing and thus convergent.
To avoid strong convergence, we must ensure that 
$\lim_{n\to\infty}\|\bx_n\| > 0$.

Now assume that $(\bu(t))_{t\in\RP}$ is a curve of unit vectors in $\bX$ 
that weakly converges to $\bzero$ as $t\to\infty$.
We make the ansatz 
\begin{equation}
\bx_n = \rho_n \bu(t_n),
\end{equation}
where $(t_n)_{n\geq 1}$ increases to $+\infty$ 
and $(\rho_n)_{n\geq 1}$ decreases to a $\rho_\infty>0$.
Then $\bx_n\weakly \bzero$, and 
$\|\bx_n\| = \rho_n \to \rho_\infty > 0$; consequently, 
$\bx_n\not\to \bzero$.

To make progress, we solve \cref{e:ant1} with equality: 
$\|\bx_{n+1}\|^2 = \scal{\bx_n}{\bx_{n+1}}$ turns into 
$\rho_{n+1}^2 = \rho_n \rho_{n+1} \scal{\bu(t_n)}{\bu(t_{n+1})}$ or 
\begin{equation}
\label{e:ant2}
\rho_{n+1} = \rho_n \scal{\bu(t_n)}{\bu(t_{n+1})}. 
\end{equation}
Following the machine learning literature, we can think of 
$\bu(t)$ as a feature map, with kernel $K$: 
\begin{equation}
\label{e:ant3}
\scal{\bu(s)}{\bu(t)} = K(s,t).
\end{equation}
Then \cref{e:ant2} reads 
\begin{equation}
\rho_{n+1} = \rho_n K(t_n,t_{n+1}).
\end{equation}
We want the kernel $K$ to satisfy 
$K(t_n,t_{n+1})<1$ (to model the decrease of $\rho_n$) yet 
\begin{equation}
\prod_{n=1}^\infty K(t_n,t_{n+1}) > 0 \quad\text{(to model $\rho_\infty>0$)}.
\end{equation}
Suppose further that the kernel $K$ is translation invariant, i.e., 
$K(s,t) = \Phi(|s-t|)$ for some function $\Phi\colon\RP\to\RP$. 
Writing $d_n := t_{n+1}-t_n$, we thus want $\sum_{n=1}^\infty d_n = +\infty$ 
(to guarantee $t_n\to\infty$) yet
\begin{equation}
\prod_{n=1}^\infty \Phi(d_n) > 0;
\end{equation}
equivalently, 
\begin{equation}
\sum_{n=1}^\infty -\ln \Phi(d_n) < +\infty.
\end{equation}
The simplest way to achieve this is to assume that 
$-\ln \Phi(d_n) = d_n^2$, i.e., $\Phi(d_n) = \exp(-d_n^2)$, 
with $\sum_{n=1}^\infty d_n^2 < +\infty$ but $\sum_{n=1}^\infty d_n = +\infty$ 
which is possible when we consider the harmonic series.
This leads us to 
$\Phi(d) = \exp(-d^2)$, and so we are led to the Gaussian kernel 
(see also \cite{SHS})
\begin{equation}
K(s,t) = \exp(-(s-t)^2).
\end{equation}
Then \cref{e:ant3} turns into 
\begin{equation}
\label{e:ant4}
\scal{\bu(s)}{\bu(t)} = \exp(-(s-t)^2).
\end{equation}

In \cref{s:constr}, we will present two realizations for $(\bu(t))_{t\in\RP}$ that 
satisfy the inner product condition~\cref{e:ant4}. 

\section{Constructing the curve $\bu$}

\label{s:constr}

In this section, we present two realizations of a curve $\bu$ satisfying the inner product condition~\cref{e:ant4}, i.e., 
\begin{equation}
\scal{\bu(s)}{\bu(t)} = \exp(-(s-t)^2).
\end{equation}

\subsection{A realization based on an orthonormal sequence}
\label{ss:ons}

Suppose that $(\mathbf{e}_k)_{\kkk}$ is an orthonormal sequence 
in $\bX$, and that $(\mathbf{e}_k)_{\kkk}$ is part of an orthonormal basis of $\bX$.
Now define 
\begin{equation}
(\forall t\in\RP)\quad \bu(t) := \sum_{k=0}^\infty u_k(t)\mathbf{e}_k, 
\quad \text{where}~ u_k(t) = \exp(-t^2)\sqrt{\frac{2^k}{k!}}t^k.
\end{equation}
Then 
\begin{align*}
\scal{\bu(s)}{\bu(t)} & = \sum_{k=0}^\infty u_k(s)u_k(t) = 
\sum_{k=0}^\infty \exp(-s^2)\sqrt{\frac{2^k}{k!}}s^k \exp(-t^2)\sqrt{\frac{2^k}{k!}}t^k\\
&= 
\exp(-s^2)\exp(-t^2)\sum_{k=0}^\infty \frac{(2st)^k}{k!}
= \exp(-s^2)\exp(-t^2)\exp(2st)\\
&= \exp(-(s-t)^2),
\end{align*}
as desired. 

The most concrete\footnote{
One can also consider $\bX=L_2[-1,1]$, 
with $\mathbf{e}_k$ being a suitably normalized 
Legendre polynomial of degree $k$ (see \cite[Section~3.7-1]{Kreyszig}), 
or $\bX=L_2[-\pi,\pi]$ with suitably normalized trigonometric functions 
(see \cite[Example~3.4.17]{DM}), 
or $\bX=L_2[0,1]$ with Walsh functions (see \cite[Example~3.4.19]{DM}). 
} version is $\bX=\ell_2$, with $\mathbf{e}_k$ being the $k$th standard
unit vector.

\subsection{A realization in $L^2(\RR)$}

\label{ss:L2}

Suppose that $\bX=L^2(\RR)$, and define $\bu(t)$ by 
\begin{equation}
\bu(t) \colon \RR\to\RR\colon r \mapsto \frac{1}{\sqrt[4]{\pi}}
\exp\Big(-\frac{(r-2t)^2}{2}\Big),
\end{equation}
which is a normalized translated Gaussian, with center at $2t$. 

Using the substitution $u = r - s - t$, one obtains 
\begin{align*}
    \langle \bu(s),\bu(t)\rangle
    &=
    \frac{1}{\sqrt{\pi}}
    \int_{\mathbb R}
    \exp\Big(-\frac{(r-2s)^2}{2}\Big)
    \exp\Big(-\frac{(r-2t)^2}{2}\Big)\,dr  \\
    &=
    \frac{1}{\sqrt{\pi}}
    \int_{\mathbb R}
    \exp\Big(
        -\frac{(r-2s)^2+(r-2t)^2}{2}
    \Big)\,dr\\
    &=
    \frac{1}{\sqrt{\pi}}
    \int_{\mathbb R}
    \exp\big(
        -(r-s-t)^2-(s-t)^2
    \big)\,dr\\
    &= 
    \frac{1}{\sqrt{\pi}}
    \exp(-(s-t)^2)
    \int_{\mathbb R}
    \exp\big(
        -(r-s-t)^2
    \big)\,dr\\
    &=
    \frac{1}{\sqrt{\pi}}
    \exp(-(s-t)^2)
    \int_{\mathbb R}
    \exp\big(
        -u^2
    \big)\,du\\
    &=
    \exp(-(s-t)^2).
\end{align*}

While $\bX=L^2(\RR)$ is separable, this construction 
is more elegant because it does not rely on working with a given orthonormal 
Schauder basis\footnote{An interesting Schauder basis for $L^2(\RR)$ consists of the 
(suitably normalized) Hermite functions \cite[Section~3.7-2]{Kreyszig}.} and the expansion of the exponential function.

\section{Auxiliary results}

\label{s:aux}

We first prove a technical inequality for future use. 

\begin{proposition}
\label{p:ineq}
Let $0 \leq x \leq 1/16$. Then 
\begin{equation}
\exp(2x+16x^2) \leq 1+32x.
\end{equation}
%with equality if and only if $x=0$.
\end{proposition}
\begin{proof}
Note that 
$80x + 512x^2 \leq 80/16 + 512(1/16)^2 = 7<30$. 
Hence $80x^2+512x^3 \leq 30x$, and so 
\begin{equation}
\label{260515a}
(2x+16x^2)(1+32x) = 2x + 80x^2 + 512x^3 \leq 2x + 30x = 32x.
\end{equation}
Set $w := 2x+16x^2$. 
Then $0\leq w \leq 2/16+16(1/16)^2 = 3/16<1$ and we learn from 
\cref{260515a} that $w(1+32x) \leq 32x$, which we re-arrange to 
the 
\begin{equation}
\label{e:steak1}
\frac{1}{1-w} \leq 1 + 32x.
\end{equation}
On the other hand, since $0\leq w < 1$, we have 
\begin{equation}
\label{e:steak2}
\exp(2x+16x^2) = \exp(w) = \sum_{n=0}^\infty \frac{w^n}{n!} \leq \sum_{n=0}^\infty w^n =
\frac{1}{1-w}. 
\end{equation}
Combining \cref{e:steak1} with \cref{e:steak2}, we obtain the conclusion. 
\end{proof}

The next two results will be useful in later sections. 

\begin{lemma}
\label{p:expexp}
Suppose that two real numbers $x,y$ satisfy 
\begin{equation}
0 \leq x \leq \frac{1}{16}\quad\text{and} \quad 
x+16x^2 \leq y. 
\end{equation}
Then 
\begin{equation}
\exp(x)+\exp(-y) \leq 2. 
\end{equation}
\end{lemma}
\begin{proof}
Because $\exp(-\cdot)$ is decreasing, it is clear that 
$\exp(x)+\exp(-y) \leq \exp(x)+\exp(-x-16x^2)$. 
It thus suffices to show that 
\begin{equation*}
g(x) := \exp(x)+\exp(-x-16x^2) \leq 2.
\end{equation*}
Note that $g'(x) = \exp(x) - (1+32x)\exp(-x-16x^2)$.
Hence
\begin{equation*}
g'(x) \leq 0 \Leftrightarrow \exp(2x+16x^2) \leq 1 + 32x,
\end{equation*}
which holds true by \cref{p:ineq}.
Therefore, $g$ is decreasing on $[0,1/16]$, and so $g(x) \leq g(0) = 2$, 
and we're done. 
\end{proof}

\begin{lemma}
\label{l:esumineq}
Let $0\leq s_1 < s_2 <\cdots < s_n$ be real numbers that 
are separated at least by some $\alpha\in\left]0,1\right]$: 
\begin{equation}
(\forall i\in\{1,\ldots,n-1\})\quad s_{i+1}-s_i \geq \alpha.
\end{equation}
Then
\begin{equation}
\max_{i\in\{1,\ldots,n\}} \sum_{j=1}^n \exp\big(-(s_i-s_j)^2\big) 
\leq \frac{1+\sqrt{\pi}}{\alpha}.
\end{equation}
\end{lemma}
\begin{proof}
Let $i\in\{1,\ldots,n\}$ be fixed. 
If $j=i\pm k$, then $|s_j - s_i| \geq k\alpha$, and so
\begin{equation}
\sum_{j=1}^n \exp\big(-(s_i-s_j)^2\big) 
\leq 1 + 2\sum_{k=1}^\infty \exp\big(-(k\alpha)^2\big).
\end{equation}
On the other hand, a Riemann-sum argument 
and \cite[Proposition~6.33]{Knapp} show that 
\begin{align}
\sum_{k=1}^\infty \exp\big(-(k\alpha)^2\big) 
&\leq 
\int_{0}^\infty \exp\big(-(t\alpha)^2\big) dt
= \frac{1}{\alpha}\int_{0}^\infty \exp\big(-u^2\big) du
= \frac{\sqrt{\pi}}{2\alpha}.
\end{align}
Altogether, we have 
\begin{equation}
\sum_{j=1}^n \exp\big(-(s_i-s_j)^2\big) 
\leq 1 + \frac{\sqrt{\pi}}{\alpha}
\leq \frac{1+\sqrt{\pi}}{\alpha},
\end{equation}
which implies the conclusion.
\end{proof}

\section{The curve and the mesh}

\label{s:curvmesh}

This section presents properties of the curve and introduces the mesh.
The two objects are central in the construction in the next section. 

\subsection{The curve}

For the rest of this paper, we assume that we are given a curve of 
unit vectors that satisfies 
\begin{empheq}[box=\mybluebox]{equation}
\label{e:Kurve}
\bu \colon \RP \to \bX
\quad\text{and}\quad
(\forall s,t\in\RP)\;\; \scal{\bu(s)}{\bu(t)} = \exp(-(s-t)^2). 
\end{empheq}
See \cref{s:constr} for realizations of such a curve.

\begin{proposition}
We have 
\begin{equation}
\label{e:david5}
\bu(t)\weakly \bzero \quad\text{as}\quad t\to\infty
\end{equation}
and 
\begin{equation}
\label{e:david6}
(\forall s,t \in\RP)\quad \norm{\bu(s)-\bu(t)}^2 = 2\big(1-\exp(-(s-t)^2)\big).
\end{equation}
\end{proposition}
\begin{proof}
Set $\bY := \cspan\{\bu(t)\}_{t\in\RP}$, 
and let $\by\in \bY$ and $\bz\in \bY^\perp$. 
For fixed $s\in\RP$, we have $\lim_{t\to\infty} \scal{\bu(s)}{\bu(t)} = 
\lim_{t\to\infty} \exp(-(s-t)^2) = 0$. 
This implies that $\scal{\by}{\bu(t)} \to 0$ as $t\to\infty$, 
while $\scal{\bz}{\bu(t)}\equiv 0$. Thus 
$\scal{\by+\bz}{\bu(t)} \to 0$ as $t\to\infty$, and 
\cref{e:david5} follows.
Now let $s,t$ be in $\RP$. 
Then 
$\norm{\bu(s)-\bu(t)}^2 = \norm{\bu(s)}^2 + \norm{\bu(t)}^2 - 2\scal{\bu(s)}{\bu(t)} = 2(1-\exp(-(s-t)^2))$.
\end{proof}

\begin{remark}
It is clear from \cref{e:david6} that the curve $\bu$ is continuous. 
One can show\footnote{We omit the details because these formulas are not used
in the rest of the paper.}
 that $\bu$ is even continuously differentiable with 
\begin{equation}
\scal{\bu'(t)}{\bu(t)} = 0, 
\;\;
\scal{\bu'(s)}{\bu'(t)} = \big(2-4(s-t)^2\big)\exp(-(s-t)^2),
\;\;
\|\bu'(t)\| = \sqrt{2},
\end{equation}
and 
\begin{equation}
\|\bu'(s)-\bu'(t)\|^2 = 4-2\big(2-4(s-t)^2\big)\exp(-(s-t)^2).
\end{equation}
In the context of \cref{ss:ons}, we have 
\begin{equation}
\bu'(t) := \sum_{k=0}^\infty v_k(t)\mathbf{e}_k, 
\quad \text{where}~ v_k(t) = \exp(-t^2)\sqrt{\frac{2^k}{k!}}\big( kt^{k-1}-2t^{k+1} \big),
\end{equation}
where for $k=0$ we set $kt^{k-1} := 0$.
In the context of \cref{ss:L2}, we have
\begin{equation}
\bu'(t) \colon \RR\to\RR\colon r \mapsto \frac{1}{\sqrt[4]{\pi}}
2(r-2t)
\exp\Big(-\frac{(r-2t)^2}{2}\Big). 
\end{equation}
\end{remark}

\subsection{The mesh}

From now on, we assume that $(d_n)_{n\geq 1}$ is a fixed sequence of 
\emph{step sizes}, or \emph{mesh increments}, such that 
\begin{empheq}[box=\mybluebox]{equation}
\label{e:dmesh00}
(\forall n\geq 1)\quad d_n>0, \;\; 
\sum_{k=1}^\infty d_k^2 < +\infty, \;\;
\text{and}\;\;
t_n := \sum_{k=1}^{n-1} d_k \to +\infty. 
\end{empheq}
Hence $(d_n)_{n\geq 1} \in \ell^2\smallsetminus\ell^1$, 
$t_1 = 0$, and $(t_n)_{n\geq 1}$ is the \emph{mesh sequence} 
we place over the nonnegative reals. 
We will use repeatedly that
\begin{equation}
\label{e:tndn}
(\forall n\geq 1)\quad t_{n+1} - t_n = d_n.  
\end{equation}
Note that we could also have started with the mesh sequence $(t_n)_{n\geq 1}$, 
and obtained the step sizes via $d_n := t_{n+1}-t_n$. 
We will also assume from now on that 
\begin{subequations}
\label{e:dmesh11}
\begin{empheq}[box=\mybluebox]{equation}
\label{e:dmesh0}
d_1 \leq \tfrac{1}{8}
\end{empheq}
and 
\begin{empheq}[box=\mybluebox]{equation}
\label{e:dmeshnew}
(\forall n\geq 1)\quad 
d_{n+1} \leq \frac{d_n}{1+64d_n^2}.
\end{empheq}
\end{subequations}
Consequently, $(d_n)_{n\geq 1}$ is strictly decreasing.
Note that \cref{e:dmeshnew} is equivalent to 
\begin{equation}
\label{e:dmeshold}
(\forall n\geq 1)\quad \frac{1}{d_{n+1}}-\frac{1}{d_n} \geq 64d_n = 64(t_{n+1}-t_n),
\end{equation}
as well as to 
\begin{equation}
\label{e:dntninc}
(\forall n\geq 1)\quad \frac{1}{d_{n}} - 64t_n \leq \frac{1}{d_{n+1}} - 64t_{n+1};
\quad\text{consequently,
$\Big(\frac{1}{d_n} - 64t_n\Big)_{n\geq 1}$ is increasing}.
\end{equation}
Moreover, if one of the inequalities in 
\cref{e:dmeshnew},\cref{e:dmeshold},\cref{e:dntninc} is an equality, 
then so are the other two.

\begin{example}[harmonic mesh]
\label{ex:harmesh}
Let $\delta \in \left]0,1/8\right]$ and suppose that 
\begin{equation}
(\forall n\geq 1)\quad d_n = \frac{\delta}{n}. 
\end{equation}
Then \cref{e:dmesh00} and \cref{e:dmesh11} hold. 
\end{example}
\begin{proof}
Clearly, each $d_n>0$, 
$\sum_{k=1}^\infty d_k^2 = {\delta^2}\sum_{k=1}^\infty \frac{1}{k^2} < +\infty$ 
while $t_n=\delta H_{n-1}\to +\infty$ where 
$H_n$ denotes the $n$th harmonic number. 
Hence \cref{e:dmesh00} holds.
It is clear that \cref{e:dmesh0} holds.
Now let $n\geq 1$. 
Then 
\begin{align*}
\frac{1}{d_{n+1}}-\frac{1}{d_n}
&= \frac{n+1}{\delta} - \frac{n}{\delta} = \frac{1}{\delta} 
\geq 64 \delta \geq 64 \cdot \frac{\delta}{n} = 64d_n.  
\end{align*}
Hence \cref{e:dmeshold} holds, and so does 
the equivalent \cref{e:dmeshnew}. 
\end{proof}

\begin{lemma}
The following hold:
\begin{equation}
\label{e:dmesh1}
(\forall 1\leq m<n)\quad \frac{1}{d_n}-\frac{1}{d_m} \geq 64(t_n-t_m), 
\end{equation}
\begin{equation}
\label{e:dmesh2}
(\forall n\geq 1)\quad d_nt_{n+1}  < \frac{1}{32},
\end{equation}
\begin{equation}
\label{e:meshcool}
(\forall 1\leq m<n)\quad 
\exp\big(2d_n(t_{n+1}-t_{m+1})\big)+ \exp\big(-2d_m(t_{n+1}-t_{m+1})\big) \leq 2.
\end{equation}
\end{lemma}
\begin{proof}
\cref{e:dmesh1}: 
Let $1\leq m < n$. 
Then \cref{e:dmeshold}, which is equivalent to \cref{e:dmeshnew}, implies that
\begin{align*}
64(t_n-t_m)&= \sum_{k=m}^{n-1} 64d_k 
\leq \sum_{k=m}^{n-1} \Big(\frac{1}{d_{k+1}}-\frac{1}{d_{k}}\Big)
= \frac{1}{d_n}-\frac{1}{d_m}.  
\end{align*}

\cref{e:dmesh2}: 
Let $n\geq 1$.  
From \cref{e:dntninc}, which is equivalent to \cref{e:dmeshnew}, we know that 
$(1/d_n - 64t_n)_{n\geq 1}$ is increasing. 
Recalling \cref{e:dmesh0}, we have in particular that
$8 \leq 1/d_1 -0 = 1/d_1 -64t_1 \leq 1/d_n -64t_n$. 
So $64t_n<1/d_n$ and $8\leq 1/d_n$. Therefore, 
\begin{equation*}
d_nt_{n+1} = d_n(t_n + d_n) 
= d_nt_n + d_n^2
< \frac{1}{64} + \frac{1}{8^2} = \frac{1}{32}. 
\end{equation*}

\cref{e:meshcool}: 
Let $1\leq m<n$ and abbreviate $\delta := t_{n+1}-t_{m+1}<t_{n+1}$. 
Then \cref{e:dmeshnew} implies that
\begin{equation}
\label{e:260516c} 
d_n<d_m; 
\end{equation}
thus, 
\begin{equation*}
\delta = t_{n+1}-t_{m+1} = t_n-t_m+(t_{n+1}-t_n) - (t_{m+1}-t_m) 
= t_n-t_m+d_n - d_m < t_n-t_m. 
\end{equation*}
Combining this with \cref{e:dmesh1} yields
\begin{equation}
\label{e:260516b}
\frac{1}{d_n}-\frac{1}{d_m} > 64\delta.
\end{equation}
Now set 
\begin{equation*}
\xi := 2d_n\delta \quad\text{and}\quad \eta := 2d_m\delta.
\end{equation*}
Because $\delta<t_{n+1}$ and \cref{e:dmesh2} hold, we have 
\begin{subequations}
\label{e:260516dd}
\begin{equation}
0 < \xi < 2d_nt_{n+1} \leq \frac{1}{16}.
\end{equation}
Using \cref{e:260516c} and \cref{e:260516b}, we estimate 
\begin{align}
\eta-\xi 
&= 2\delta(d_m-d_n) = 2\delta d_nd_m\Big(\frac{1}{d_n}-\frac{1}{d_m}\Big)
> 2\delta d_n^2 \cdot 64\delta = 32 \cdot 4d_n^2\delta^2 = 32\xi^2 
> 16\xi^2.
\end{align}
\end{subequations} 
Combining \cref{e:260516dd} with \cref{p:expexp} yields 
$\exp(\xi) + \exp(-\eta) \leq 2$, as announced. 
\end{proof}

\section{The sequence}

\label{s:seq}

As explained in \cref{s:motiv}, 
we now define the sequence of scalars $(\rho_n)_{n\geq 1}$ by 
\begin{empheq}[box=\mybluebox]{equation}
\label{e:rho}
\rho_1 := 1 \;\;\text{and}\;\;
\rho_{n+1} := \rho_n\exp(-d_n^2).
\end{empheq}

\begin{proposition}
\label{p:rho}
The sequence $(\rho_n)_{n\geq 1}$ satisfies 
\begin{equation}
\label{e:rho1}
(\forall n\geq 1)\quad \rho_n = \exp\Big(-\sum_{k=1}^{n-1} d_k^2\Big);
\end{equation}
moreover, 
\begin{empheq}[box=\mybluebox]{equation}
\label{e:rho2}
\rho_n > \rho_{n+1}\to \rho_\infty := \exp\Big(-\sum_{k=1}^\infty d_k^2\Big) > 0
\;\;\text{ as } n\to\infty.
\end{empheq}
\end{proposition}
\begin{proof}
Because each $d_n>0$, it is clear that $\rho_{n+1}<\rho_n$. 
The identity \cref{e:rho1} follows by induction. 
Since $\sum_{k=1}^\infty d_k^2 < +\infty$, we have $\rho_\infty > 0$ and $\rho_n \to \rho_\infty$ as $n\to\infty$.
\end{proof}

We are now ready to define the key sequence $(\bx_n)_{n\geq 1}$ in this paper, 
namely 
\begin{empheq}[box=\mybluebox]{equation}
\label{xn}
(\forall n\geq 1)\quad \bx_n := \rho_n \bu(t_n) \in \bX. 
\end{empheq}
We also consider the smallest closed convex cone containing the 
corresponding set $\{\bx_n\}_{n\geq 1}$: 
\begin{empheq}[box=\mybluebox]{equation}
\label{bK}
\bK := \ccone\conv\{\bx_n\}_{n\geq 1} = 
\cconv \RP \{\bx_n\}_{n\geq 1} \subseteq \bX.
\end{empheq}

\begin{proposition}
\label{kurve}
We have 
\begin{equation}
\label{e:kurve}
\bx_n \weakly \bzero \;\;\text{and}\;\; \norm{\bx_n} =\rho_n \to \rho_\infty > 0. 
\end{equation}
Let $1\leq m\leq n$. Then the following hold: 
\begin{enumerate}[label=(\roman*), font=\normalfont]
\item 
\label{kurve1}
$\scal{\bx_m}{\bx_n} = \rho_m\rho_n\exp(-(t_n-t_m)^2)$ 
and $\|\bx_n\|=\rho_n > \rho_{n+1} = \|\bx_{n+1}\|$. 
\item 
\label{kurve4}
$\scal{\bK}{\bK} \geq 0$, and so $\bK$ is an acute cone, i.e., 
$\bK \subseteq \bK^{\oplus}$.
\item 
\label{kurve2}
$\scal{\bx_{n+1}}{\bx_n} = \|\bx_{n+1}\|^2$. 
\item 
\label{kurve3}
$\scal{\bx_{n+1}-\bx_n}{\bx_{n+1}} = 0$.
\end{enumerate}
\end{proposition}
\begin{proof}
\cref{e:kurve} follows from 
\cref{xn}, \cref{e:david5}, and \cref{e:rho2}.
\cref{kurve1}: Clear from \cref{xn} and \cref{e:Kurve}. 
\cref{kurve4}: 
This follows from \cref{kurve1}. 
\cref{kurve2}: 
Indeed, \cref{kurve1} and \cref{e:rho} yield 
\begin{equation*}
\scal{\bx_{n+1}}{\bx_n} = \rho_{n+1}\rho_n\exp(-(t_{n+1}-t_n)^2) = \rho_{n+1}\rho_n\exp(-d_n^2)
=\rho_{n+1}^2   = \|\bx_{n+1}\|^2.  
\end{equation*}
\cref{kurve3}: \cref{kurve2} implies 
$\scal{\bx_{n+1}-\bx_n}{\bx_{n+1}} = \scal{\bx_{n+1}}{\bx_{n+1}} - \scal{\bx_n}{\bx_{n+1}} = 
\|\bx_{n+1}\|^2 - \|\bx_{n+1}\|^2 = 0$. 
\end{proof}

For later convenience, and also motivated by \cref{e:kurve}, 
we denote the weak limit of $(\bx_n)_{n\geq 1}$ by 
\begin{empheq}[box=\mybluebox]{equation}
\label{e:bxinf}
\bx_\infty := \bx_{\infty+1} := \bzero\in \bX. 
\end{empheq}

\begin{lemma}
\label{l:schap}
%Suppose that \cref{e:meshcool} holds. 
We have 
\begin{equation}
\label{e:schap}
(\forall 1\leq m < n \leq \infty)\quad
0\leq \scal{\bx_{m+1}-\bx_{n+1}}{(\bx_{m}-\bx_{m+1})-(\bx_{n}-\bx_{n+1})}. 
\end{equation}
\end{lemma}
\begin{proof}
Let $1\leq m < n \leq \infty$.
We will argue by cases. In both cases, we will work backward from the desired conclusion \cref{e:schap} 
by reversible algebraic manipulations.

\emph{Case~1:} $n=\infty$. 
Then $\bx_{n+1} = \bx_\infty = \bzero$ (recall \cref{e:bxinf}), 
and so 
\cref{e:schap} reduces to 
$0\leq\scal{\bx_{m+1}}{\bx_{m}-\bx_{m+1}}$, which holds 
with equality by \cref{kurve}\cref{kurve3}.

\emph{Case~2:} $n<\infty$. 
Using again \cref{kurve}\cref{kurve3}, we see that \cref{e:schap} is equivalent to
\begin{equation*}
0 \leq - \scal{\bx_{m+1}}{\bx_{n}-\bx_{n+1}} - \scal{\bx_{n+1}}{\bx_{m}-\bx_{m+1}},
\end{equation*}
and hence also to 
\begin{equation*}
0 \leq 2\scal{\bx_{m+1}}{\bx_{n+1}} - \scal{\bx_{m+1}}{\bx_n} - \scal{\bx_{n+1}}{\bx_m}.
\end{equation*}
In view of \cref{kurve}\cref{kurve1}, this is equivalent to
\begin{equation*}
0 \leq 2\rho_{m+1}\rho_{n+1}\exp(-(t_{n+1}-t_{m+1})^2) - \rho_{m+1}\rho_n\exp(-(t_n-t_{m+1})^2) - \rho_{n+1}\rho_m\exp(-(t_{n+1}-t_m)^2).
\end{equation*}
In turn, using \cref{e:rho}, we can rewrite this condition as 
\begin{subequations}
\label{e:david1}
\begin{align}
0 &\leq 
2\rho_{m}\exp(-d_m^2)\cdot \rho_{n}\exp(-d_n^2)\cdot\exp(-(t_{n+1}-t_{m+1})^2) \\
&\quad - \rho_{m}\exp(-d_m^2)\cdot \rho_n\cdot \exp(-(t_n-t_{m+1})^2) \\
&\quad - \rho_{n}\exp(-d_n^2)\cdot \rho_m\cdot \exp(-(t_{n+1}-t_m)^2). 
\end{align}
\end{subequations}
Abbreviate $\delta := t_{n+1}-t_{m+1}$ and 
recall that $t_{n+1} = t_n + d_n$ (see \cref{e:tndn}). 
Then $t_n-t_{m+1} = \delta - d_n$ and 
$t_{n+1}-t_m = \delta + d_m$. 
Hence \cref{e:david1} is equivalent to
\begin{equation*}
0 \leq \rho_m\rho_n
\big(2\exp(-d_m^2-d_n^2-\delta^2) - \exp(-d_m^2-(\delta-d_n)^2) 
- \exp(-d_n^2-(\delta+d_m)^2)\big), 
\end{equation*}
and, after dividing by $\rho_m\rho_n>0$ and expanding the squares, to 
\begin{align*}
0 
&\leq 
2\exp(-d_m^2-d_n^2-\delta^2) - \exp(-d_m^2-\delta^2 + 2d_n\delta-d_n^2)
-\exp(-d_n^2-\delta^2 - 2d_m\delta-d_m^2) \\
&= \exp(-d_m^2-d_n^2-\delta^2)
\big(2 - \exp(2d_n\delta) - \exp(-2d_m\delta)\big).
\end{align*}
Dividing by $\exp(-d_m^2-d_n^2-\delta^2)>0$, we now face 
\begin{equation*}
0 \leq 2 - \exp(2d_n\delta) - \exp(-2d_m\delta);
\end{equation*}
however, this is just a re-arrangement of \cref{e:meshcool}. 
\end{proof}

\begin{lemma}
\label{l:david2}
Let $\by\in \bK$, and suppose that 
\begin{align}
\label{e:david3}
(\forall 1\leq n< \infty)\quad 
0 \leq \scal{\bx_{n+1}-\by}{\bx_{n}-\bx_{n+1}}.
\end{align}
Then $\by=\bzero$. 
\end{lemma}
\begin{proof}
In view of \cref{kurve}\cref{kurve3}, we see that \cref{e:david3} is equivalent to
\begin{equation}
\label{e:david4}
(\forall 1\leq n< \infty)\quad
\scal{\bx_n}{\by} \leq \scal{\bx_{n+1}}{\by},
\end{equation}
i.e., the sequence $(\scal{\bx_n}{\by})_{n\geq 1}$ is increasing.
Moreover, $\scal{\bx_n}{\by}\to 0$ by \cref{e:kurve}. 
Combining these two observations yields 
$(\forall n\geq 1)$ $\scal{\bx_n}{\by} \leq 0$; 
hence (recall \cref{bK}),
\begin{equation*}
\by \in \bK^\ominus.
\end{equation*}
On the other hand, $\by\in\bK$ by assumption. 
Altogether, $0\leq \|\by\|^2 = 
\scal{\by}{\by} \leq 0$ and we're done. 
\end{proof}

\begin{lemma}
\label{l:Cesaro}
Let $n\geq 1$, and let $I$ be a nonempty subset of $\{1,\ldots,n\}$. 
Set $\delta := \max\menge{|t_i-t_j|}{i,j\in I}$. 
Then
\begin{equation}
\bigg\|\frac{1}{n}\sum_{k=1}^n\bx_k \bigg\|
\geq \frac{\# I}{n}\rho_\infty \exp(-\delta^2/2)>0.
\end{equation}
\end{lemma}
\begin{proof}
By \cref{kurve}, we have 
\begin{equation*}
(\forall i,j\in I)\quad 
\scal{\bx_i}{\bx_j} \geq \rho_\infty^2\exp(-\delta^2) > 0.
\end{equation*}
Hence 
\begin{align*}
\bigg\|\frac{1}{n}\sum_{k=1}^n\bx_k \bigg\|^2
&= \frac{1}{n^2}\sum_{i,j=1}^n \scal{\bx_i}{\bx_j}
\geq  \frac{1}{n^2}\sum_{i,j\in I} \scal{\bx_i}{\bx_j}
\geq \frac{(\# I)^2}{n^2}\rho_\infty^2\exp(-\delta^2).
\end{align*}
Now take the square root, and we're done.
\end{proof}

\section{Main results}

\label{s:main}

\subsection{Firmly nonexpansive iterates that do not converge strongly}
We are now ready for our first main result, concerning regular iterates. 

\begin{theorem}[a firmly nonexpansive iteration that converges weakly but not strongly] 
\label{t:main1}
Set 
\begin{equation}
\bC := \cconv\{\bx_n\}_{n\geq 2}. 
\end{equation}
Then there exists a firmly nonexpansive operator 
\begin{equation}
\bT\colon \bX\to \bX 
\end{equation}
such that 
\begin{equation}
\cran \bT =\bC,\;\; 
\Fix \bT = \{\bzero\}, \;\;
(\forall n\geq 1)\; \bT\bx_n = \bx_{n+1},
\end{equation}
and 
\begin{equation}
\bT^n\bx_1\weakly \bzero, \quad\text{ but } \quad
\inf_{1\leq n<\infty} \|\bT^n\bx_1\|=\rho_\infty >0.
\end{equation}
\end{theorem}
\begin{proof}
Since $(\bx_n)_{n\geq 2}$ lies in $\bC$, which is weakly closed, 
and since $\bx_n \weakly \bzero$ (by \cref{e:kurve}), 
we obtain $\bzero\in \bC$.
Now set 
\begin{equation*}
\bT_0 \colon  \{\bx_n\}_{n\geq 1}\cup\{\bx_\infty\} \to
\{\bx_n\}_{n\geq 2}\cup\{\bx_\infty\}\colon \bx_n\mapsto \bx_{n+1}. 
\end{equation*}
Then $\bT_0$ is a well-defined bijection 
(recall \cref{kurve}\cref{kurve1}), 
with $\Fix \bT_0 = \{\bx_\infty\} = \{\bzero\}$. 
Combining \cref{l:schap} with \cite[Proposition~4.4]{BC2017}, 
we see that 
\begin{equation*}
\text{$\bT_0$ is firmly nonexpansive.}
\end{equation*}
By a refined version of the Kirszbraun-Valentine theorem 
(see \cite[Corollary~5]{Bauschke07}), 
there exists a firmly nonexpansive extension 
$\bT\colon \bX \to \bX$ 
of $\bT_0$, with the extra range localization property 
$\ran \bT \subseteq \cconv \ran \bT_0 = \bC$.
Because $\bT$ extends $\bT_0$, we clearly have 
$\ran \bT_0 \subseteq \ran \bT$ and so $\cran \bT_0 \subseteq \cran \bT$. 
Because $\bT$ is firmly nonexpansive, hence maximally monotone, it follows that $\cran \bT$ is convex 
(see, e.g., \cite[Corollary~21.14]{BC2017}). Altogether,
\begin{equation*}
\cran \bT = \cconv \ran \bT_0 = \bC.
\end{equation*}
Clearly, $\{\bzero\} = \Fix \bT_0 \subseteq \Fix \bT$.
Conversely, let $\by \in \Fix \bT$. 
Then $\by \in \cran \bT = \bC \subseteq \bK$. 
Because $\bT$ is firmly nonexpansive, we must have that 
\cref{e:david3} holds. 
Hence \cref{l:david2} implies that $\by=\bzero$. 
Altogether, $\Fix \bT = \{\bzero\}$. 
Because $\bT$ extends $\bT_0$, we have 
$\bT\bx_n = \bT_0\bx_n = \bx_{n+1}$ for all $n\geq 1$. 
Finally, \cref{e:kurve} shows that 
$\bT^n\bx_1 = \bx_{n+1} \weakly \bzero$ 
but $(\|\bT^n\bx_1\|)_{n\geq 1}$ decreases 
to $\rho_\infty>0$. 
\end{proof}

\begin{remark}[historical comments]
Several comments on \cref{t:main1} are in order.
\begin{enumerate}[label=(\roman*), font=\normalfont]
\item We already pointed out that 
Genel and Lindenstrauss \cite{GL} provided a similar 
construction (see \cref{f:GL}). 
Their proof is much more geometric, 
and it also relies on the Kirszbraun-Valentine extension theorem. 
\item Hundal \cite{Hundal} constructed a halfspace $H$ and a nonempty closed cone $K$ in $\ell^2$ such that $H\cap K\neq\varnothing$, and the sequence generated by iterating 
the composition of the projections $P_KP_H$ fails to converge strongly to $0$, the unique point in $H\cap K$.  While $P_KP_H$ is an averaged nonexpansive mapping, 
it is not firmly nonexpansive 
(see \cite[Lemma~4.3]{BMR}). 
\item 
The first proximal point iteration that fails to converge strongly --- even with flexibility in the parameters --- is due to G\"uler \cite{Guler}. 
In \cite{BMR}, Hundal's example was re-interpreted as the iteration of a proximal (hence firmly nonexpansive) mapping. 
\end{enumerate}

The proofs of the examples constructed by Genel and Lindenstrauss, 
by Hundal, and by G\"uler appear to be significantly more complicated than the proof of \cref{t:main1}. However, G\"uler's construction gives a proximal mapping, rather than just a firmly nonexpansive one. 
\end{remark}

\begin{remark}[$\bT$ viewed as a resolvent]
Consider the operator $\bT$ from \cref{t:main1}. 
Because $\bT$ is firmly nonexpansive, it must be 
the resolvent $J_{\bA} = (\bA+\Id)^{-1}$ of some maximally monotone operator 
$\bA\colon \bX \To \bX$. The Minty parametrization implies that 
\begin{equation}
\{(\bx_{n+1},\bx_n-\bx_{n+1})\}_{n\geq 1} \cup \{(\bzero,\bzero)\}
\subseteq \gra \bA.
\end{equation}
Note that $\bzero \in \bA\bzero$ and $\bx_{n}-\bx_{n+1}\in \bA\bx_{n+1}$.
By \cref{kurve}\cref{kurve3}, 
$\scal{\bx_{n+1}-\bzero}{(\bx_{n} -\bx_{n+1}) - \bzero}
= 0$. 
Hence $\bA$ is not strictly monotone. 
If $\bA$ were paramonotone, then it would follow that 
$\bzero\in \bA\bx_{n+1}$ but this is false because 
$\zer \bA = \Fix \bT = \{\bzero\}$. 
Hence $\bA$ is not paramonotone. Consequently, 
$\bA$ is not a subdifferential operator and $\bT=J_\bA$ cannot be
a proximal mapping. 
\end{remark}

\subsection{Ces\`aro means that do not converge strongly either}

We are now ready for our second main result:
the incarnation of \cref{t:main1} through the harmonic mesh produces
Ces\`aro means that do not converge strongly.

\begin{example}[harmonic mesh and Ces\`aro means]
\label{ex:harmCes}
Suppose that the given mesh increments $(d_n)_{n\geq 1}$ come from the scaled harmonic 
mesh (see \cref{ex:harmesh}): 
let $\delta \in \left]0,1/8\right]$, and assume that 
\begin{equation}
(\forall n\geq 1)\quad d_n = \frac{\delta}{n}.
\end{equation}
For the sequence $(\bx_n)_{n\geq 1}$ defined in \cref{xn}, 
we saw in \cref{t:main1} that $(\bx_n)_{n\geq 1} = 
(\bT^{n-1}\bx_1)_{n\geq 1}$ 
converges weakly to $\bzero$ but 
$\inf_{n\geq 1} \|\bT^{n-1}\bx_1\| = \rho_\infty > 0$.
In fact, we have 
\begin{equation}
\|\bx_1\| > \|\bx_2\|> \cdots > \|\bx_n\|\to  
\rho_\infty = \exp\Big(-\frac{\delta^2\pi^2}{6}\Big) > 0.
\end{equation}
Now consider the Ces\`aro means of $(\bx_n)_{n\geq 1}$, i.e., 
\begin{equation}
\by_n := \frac{1}{n}\sum_{k=1}^n \bx_k = 
\frac{1}{n}\sum_{k=1}^n \bT^{k-1}\bx_1.
\end{equation}
Then $(\by_n)_{n\geq 1}$ converges weakly to $\bzero$ as well; 
moreover, 
\begin{equation}
\inf_{n\geq 1} \|\by_n\| \geq 
\frac12 \cdot \exp\Big(-\frac{\delta^2}{6}\big({\pi^2}+3\big)\Big) 
>0. 
\end{equation}
\end{example}
\begin{proof}
\cref{e:rho2} and, e.g., \cite{Ghosh}  yield
\begin{equation}
\label{e:jackr8}
\rho_\infty = \exp\Big(-\sum_{k=1}^\infty d_k^2\Big) = 
\exp\Big(-\delta^2\sum_{k=1}^\infty \frac{1}{k^2}\Big) 
= \exp\Big(-\frac{\delta^2\pi^2}{6}\Big) > 0.
\end{equation}
We now turn to the Ces\`aro means. 
Note that $\by_1 = \bx_1$ and so  
\cref{e:rho} and \cref{kurve}\cref{kurve1} yield 
\begin{equation}
\label{e:ryobi}
\|\by_1\| = \|\bx_1\| =\rho_1 = 1. 
\end{equation}
Now let $n\geq 2$, and consider the ``upper-half'' index set 
\begin{equation*}
I_n := \big\{\lfloor n/2\rfloor+1,\ldots,n\big\}. 
\end{equation*}
Then\footnote{The formulas involving floor and ceiling functions are 
proved by discussing parity ($n$ is even or $n$ is odd).} 
$\# I_n = \lceil n/2\rceil \geq n/2$.
Let $i,j$ be in $I_n$ with $i\leq j$. 
Then $(j-1)-i+1 = j-i \leq n - (\lfloor n/2\rfloor+1)
= \lfloor (n-1)/2\rfloor$, 
and if $i\leq k\leq j-1$, then 
$\lfloor n/2\rfloor + 1 \leq k$ and so 
$1/k \leq 1/(\lfloor n/2\rfloor + 1)$.
It follows that 
\begin{align*}
t_j-t_i = \delta \sum_{k=i}^{j-1} \frac{1}{k} 
\leq \delta \cdot \lfloor (n-1)/2\rfloor \cdot \frac{1}{\lfloor n/2\rfloor + 1}
<\delta. 
\end{align*}
We now deduce from \cref{l:Cesaro} and \cref{e:jackr8} that 
\begin{equation*}
\|\by_n\| \geq \frac{\# I_n}{n}\rho_\infty \exp(-\delta^2/2)
\geq \frac{1}{2}\cdot \rho_\infty \exp(-\delta^2/2) 
= \frac{1}{2}\cdot \exp\Big(-\frac{\delta^2}{6}\big({\pi^2}+3\big)\Big) > 0,
\end{equation*}
and we're done.
\end{proof}

\begin{remark}
To the best of our knowledge, this is the first example 
of a firmly nonexpansive mapping $\bT$ such that its 
iterates $\bT^n{\bx_1}\weakly 0$, 
but even the corresponding Ces\`aro means 
$\frac{1}{n}\sum_{k=1}^n \bT^{k-1}\bx_1$ 
are bounded away from $\bzero$ for some starting point $\bx_1$.
In view of \cite[Theorem~2.1]{Wittmann}, the mapping $\bT$ is not odd. 
In \cref{s:bizarre}, we will present an even more bizarre example where 
one subsequence stays away from $\bzero$ yet another converges 
to $\bzero$ \emph{strongly}.
\end{remark}

\subsection{Ces\`aro means that do converge strongly}

Our third main result essentially states that sometimes 
Ces\`aro means behave much nicer than straight iterates in the following 
sense:

\begin{theorem}[Ces\`aro means that converge strongly]
\label{t:Cnice}
Suppose that the given mesh increments $(d_n)_{n\geq 1}$ satisfy
\begin{equation}
\label{hrv01}
nd_n\to\infty.
\end{equation}
For the sequence $(\bx_n)_{n\geq 1}$ defined in \cref{xn}, 
we saw in \cref{t:main1} that $(\bx_n)_{n\geq 1} = 
(\bT^{n-1}\bx_1)_{n\geq 1}$ 
converges weakly to $\bzero$ but 
$\inf_{n\geq 1} \|\bT^{n-1}\bx_1\| = \rho_\infty > 0$.
Then the corresponding Cesaro means satisfy 
\begin{equation}
\by_n := \frac{1}{n}\sum_{k=1}^n \bx_k = 
\frac{1}{n}\sum_{k=1}^n \bT^{k-1}\bx_1 \to 0.
\end{equation}
\end{theorem}
\begin{proof}
Recall that \cref{kurve}\cref{kurve1} gives 
$\scal{\bx_i}{\bx_j} = \rho_i\rho_j\exp(-(t_i-t_j)^2)$, 
and that we also have $\{\rho_n\}_{n\geq 1}\subseteq \left]0,1\right]$ 
by \cref{e:rho1}. 
Hence 
\begin{subequations}
\label{hrv02}
\begin{align}
\|\by_n\|^2
&=\bigg\| \frac{1}{n}\sum_{k=1}^n \bx_k\bigg\|^2
= \frac{1}{n^2}\sum_{i=1}^n\sum_{j=1}^n\scal{\bx_i}{\bx_j}
= \frac{1}{n^2}\sum_{i=1}^n\sum_{j=1}^n\rho_i\rho_j\exp\big(-(t_i-t_j)^2\big)\\
&\leq 
\frac{1}{n^2}\sum_{i=1}^n\sum_{j=1}^n\exp\big(-(t_i-t_j)^2\big).
\end{align}
\end{subequations}
The sequence $(d_n)_{n\geq 1}$ 
is strictly decreasing by \cref{e:dmesh00} and \cref{e:dmeshnew}.
Hence if $k\in\{1,\ldots,n-1\}$, then 
$t_{k+1}-{t_k} = d_k > d_n$.
Moreover, $d_1\leq 1/8$ by 
\cref{e:dmesh0}.
Applying \cref{l:esumineq} with $\alpha=d_n$, we obtain 
\begin{equation*}
\max_{i\in\{1,\ldots,n\}}\sum_{j=1}^n\exp\big(-(t_i-t_j)^2\big) 
\leq \frac{1+\sqrt{\pi}}{d_n};
\end{equation*}
therefore,
\begin{equation}
\label{hrv03}
\sum_{i=1}^n\sum_{j=1}^n\exp\big(-(t_i-t_j)^2\big) 
\leq n\frac{1+\sqrt{\pi}}{d_n}. 
\end{equation}
Combining \cref{hrv02} with \cref{hrv03}, we estimate
\begin{equation*}
\|\by_n\|^2 \leq \frac{1+\sqrt{\pi}}{nd_n},
\end{equation*}
which yields the conclusion.
\end{proof}

\begin{example}
\label{ex:Cnice}
Suppose that 
\begin{equation}
(\forall n\geq 1)\quad d_n = \frac{1}{16n^{3/4}}.
\end{equation}
Then all mesh assumptions \cref{e:dmesh00} and \cref{e:dmesh11} hold. 
Consequently,
$(\bx_n)_{n\geq1}=(\bT^{n-1}\bx_1)_{n\geq1}$ converges weakly
to $\bzero$ and
$\inf_{n\geq 1} \|\bT^{n-1}\bx_1\| = \rho_\infty > 0$.
Moreover,
the sequence of the 
corresponding Ces\`aro means of $(\bx_n)_{n\geq 1}$ satisfies 
\begin{equation}
\by_n := \frac{1}{n}\sum_{k=1}^n \bx_k = 
\frac{1}{n}\sum_{k=1}^n \bT^{k-1}\bx_1 \to 0.
\end{equation}
\end{example}
\begin{proof}
It's clear from the assignment of $(d_n)_{n\geq 1}$ that 
\cref{e:dmesh00} and \cref{e:dmesh0} hold. 
It remains to prove \cref{e:dmeshnew}; or equivalently, \cref{e:dmeshold},
which here turns into 
$16(n+1)^{3/4}-16n^{3/4} \geq 64/(16n^{3/4})$.
In turn, this is equivalent to 
\begin{equation}
\label{hrvgoal}
4(n+1)^{3/4}-4n^{3/4}\geq \frac{1}{n^{3/4}}.
\end{equation}
To see \cref{hrvgoal}, 
set $\varphi(t) := 4t^{3/4}$ for $t>0$.
Then $\varphi'(t) = 3t^{-1/4}$ and 
$\varphi''(t)=-\frac{3}{4}t^{-5/4}<0$. 
Hence $\varphi'$ is strictly decreasing. 
The Mean Value Theorem provides $\eta\in \left]n,n+1\right[$ such that
$\varphi(n+1)-\varphi(n)=\varphi'(\eta)>\varphi'(n+1)$. 
Thus, on the one hand, we obtain 
\begin{equation}
4(n+1)^{3/4}-4n^{3/4} > \frac{3}{(n+1)^{1/4}}. 
\end{equation}
On the other hand, we have the equivalences
$3/(n+1)^{1/4} \geq 1/n^{3/4}$ 
$\Leftrightarrow$ 
$3n^{3/4}\geq (n+1)^{1/4}$ 
$\Leftrightarrow$ 
$81n^3 \geq n+1$, and the last inequality is certainly true.  
Altogether, \cref{hrvgoal} holds.

The ``Consequently'' statement concerning $(\bx_n)_{n\geq 1}$ now follows from 
\cref{t:main1}. 
Finally, the ``Moreover'' statement is clear from \cref{t:Cnice}. 
\end{proof}

\section{Bizarre Ces\`aro means: a subsequence that stays away from zero
and another that doesn't}

\label{s:bizarre}

In this section, we present a construction of a mesh 
whose corresponding 
firmly nonexpansive mapping $\bT$ has somewhat bizarre 
properties. 
It all starts with our choice of block widths: 
\begin{empheq}[box=\mybluebox]{equation}
\label{e:wk}
(w_k)_{k\geq 1}  := (k+1)_{k\geq 1} = (2,3,4,\ldots).
\end{empheq}

We will iteratively construct a sequence of indices 
$(i(k))_{k\geq 1}$ with 
\begin{empheq}[box=\mybluebox]{equation}
1 =: i(1) < i(2) < \cdots,
\;\;\text{and blocks}\;\;
I_k := \{i(k), i(k)+1, \ldots, i(k+1)-1\}.
\end{empheq}

\subsection{The first block}

\label{ss:k=1n=1}

Let us now explain what happens in the first block $I_1$.
We let\footnote{Following common usage in computer science,
$I_1 \gets \{1\}$ means that $I_1$ is currently $\{1\}$ but this 
set may be updated as the construction proceeds. In contrast, 
$i(1) := 1$ signifies that $i(1)$ is defined to be $1$ and will stay
that way.}
\begin{equation}
i(1) := 1,\;\;
I_1 \gets \{1\},\;\;
Q_1 \in \NN\cap \left[8,+\infty\right[. 
\end{equation}
Note that we only know $i(1) = \min I_1$ and the eventually final block $I_1$ 
will be constructed by enlarging $I_1$. 
Right now, $I_1$ is the current candidate for the first block. 
We now consider the first index $n$ in the first candidate block: 
\begin{equation}
\label{e:firstn}
n \gets i(1) = 1\in I_1, \;\;
t_1 := 0,\;\;
d_1 := \frac{1}{Q_1} \leq \frac{1}{8}.
\end{equation}
We now set $k \gets 1$, indicating the block counter. 
We are ready to move from $n$ to $n+1$, which will be explained
in the next subsection.

\subsection{Deciding on the fate of $n+1$ given $n\in I_k$}

Now suppose that 
\begin{subequations}
\label{e:knknown}
\begin{equation}
\label{e:i(k)known}
  \text{$k$ and $Q_k$ are  given, $i(k)$ is known, $I_k$ is the current candidate block, 
  }
\end{equation}
and that 
\begin{equation}
\label{e:ninJk}
\text{$n\in I_k$ is given, and we know $t_n$ and $d_n$.}
\end{equation}
\end{subequations}
(This is definitely true for $k=1$ and $n=i(1)=1$, see 
\cref{ss:k=1n=1}.)
We always set 
\begin{equation}
\label{e:tn+1}
t_{n+1} := t_n + d_n.
\end{equation}
Next, 
we compare $t_{n+1}$ to $t_{i(k)} + w_k$.

\textbf{Case~1 (staying in the block):} 
$t_{n+1} < t_{i(k)} + w_k$.\\
(For $k=1$, this means $t_{n+1} < 2$, which is definitely true for $n=1$ 
because $t_2 = t_1+d_1 = 0 + 1/Q_1 \leq 1/8$.)
We then enlarge the candidate block and update the mesh increment 
according to 
\begin{equation}
\label{e:stayn+1}
I_{k} \gets I_k \cup \{n+1\},\;\;
d_{n+1} := \frac{d_n}{1+64 d_n^2}. 
\end{equation}
Keeping $k$ unchanged, we now update
\begin{equation}
n \gets n+1,
\end{equation}
and return to \cref{e:knknown}. 

\textbf{Case~2 (leaving the block):}
$t_{n+1} \geq t_{i(k)} + w_k$.\\
This condition signals that $n$ is the last index of $I_k$, 
and $n+1$ is the first index of $I_{k+1}$:
\begin{equation}
I_k := \{i(k),\ldots,n\} \;\;\text{is complete},\quad i(k+1) := n+1, \quad I_{k+1} \gets \{i(k+1)\}.
\end{equation}
We now update the integer parameter $Q_{k+1}$ 
by picking an integer such that 
\begin{equation}
\label{e:Qk+1}
Q_{k+1} \geq \max\bigg\{i(k+1), 2^{k+2}w_{k+1}, 
\frac{1 + 64d^2_{i(k+1)-1} }{d_{i(k+1)-1}} \bigg\}. 
\end{equation}
Note that $n+1\in I_{k+1}$, that 
we know $t_{n+1}$ (from \cref{e:tn+1}), 
but this time we update 
\begin{equation}
\label{e:Qk+1'}
d_{n+1} := \frac{1}{Q_{k+1}}. 
\end{equation}
We now update both counters 
\begin{equation}
k\gets k+1, \quad n\gets n+1
\end{equation}
before returning to \cref{e:knknown}.

We have provided well-defined update rules. 
What is not yet resolved is whether we ever leave the first block, and 
whether the so-constructed mesh satisfies all required properties. 
We tackle these questions next. 

\subsection{Escaping a block and guaranteeing that $t_n\to\infty$}

Suppose that $n$ and $n+1$ are in $I_k$. By the condition 
for being in the block 
(see \textbf{Case~1} above),
we have 
\begin{equation}t_n < t_{i(k)} + w_k.
\end{equation}
By \cref{e:stayn+1} and \cref{e:firstn}, we have 
\begin{equation}
d_{n+1} = d_n/(1+64 d_n^2) < d_n < \cdots 
< d_{i(k)} = \frac{1}{Q_k},
\end{equation}
and 
(analogously to our discussion around \cref{e:dntninc})
\begin{equation}
\frac{1}{d_{n}} -64 t_n = \frac{1}{d_{i(k)}} - 64 t_{i(k)} = Q_k - 64 t_{i(k)}.
\end{equation}
Hence altogether
\begin{equation}
\label{e:jackr0}
Q_k \leq \frac{1}{d_n} = Q_k + 64(t_n-t_{i(k)}) < Q_k + 64 w_k.
\end{equation}
Thus 
\begin{equation}
\label{e:jackr1}
0<\frac{1}{Q_k+64w_k} < d_n \leq \frac{1}{Q_k}.
\end{equation}
This lower bound, which is valid for all $n\in I_k$ and independent of $n$, 
guarantees that we escape block $k$: indeed, if $n_k$ is the smallest index in $I_k$ 
such that 
$t_{n_k+1} = t_{i(k)} + \sum_{m=i(k)}^{n_k} d_m \geq t_{i(k)}+w_k$, 
then $I_k = \{i(k),\ldots,n_k\}$, $i(k+1) = n_k+1$, and 
\begin{equation}
t_{n_k+1} \geq t_{i(k)} + w_k \geq w_k = k+1 \to\infty 
\quad\text{as $k\to\infty$.}
\end{equation}
Because $(t_n)_{n\geq 1}$ is clearly strictly increasing, 
we obtain 
\begin{equation}
\lim_{n\to\infty} t_n = \infty.
\end{equation}

\subsection{Square summability of the mesh increments}

We have, in view of \cref{e:jackr1} and \textbf{Case~1} ($i(k+1)-1\in I_k$), 
that 
\begin{equation}
w_k \leq \sum_{n\in I_k} d_n = t_{i(k+1)} - t_{i(k)} 
= d_{i(k+1)-1} + (t_{i(k+1)-1} - t_{i(k)}) 
<\frac{1}{Q_k}  + w_k.
\end{equation}
Hence if ${n}\in I_k$ and thus $d_n\leq 1/Q_k$ by \cref{e:jackr1}, we get
\begin{equation}
\sum_{n\in I_k} d_n^2 \leq \frac{1}{Q_k}\sum_{n\in I_k} d_n
< \frac{1}{Q_k^2}+ \frac{w_k}{Q_k}. 
\end{equation}
On the other hand, from \cref{e:Qk+1}, we get 
$Q_k \geq 2^{k+1}w_k \geq 2^{k+1}$ and so 
\begin{equation}
\sum_{n\in I_k} d_n^2 \leq \frac{1}{4^{k+1}} + \frac{1}{2^{k+1}}.
\end{equation}
Summing this over $k$ gives the summability condition\footnote{The geometric series upper bound can be used to obtain a positive lower bound for $\rho_\infty$ if needed.} 
$\sum_n d_n^2 <\infty$.

\subsection{The mesh satisfies all assumptions}

Combining the previous two subsections, we've verified \cref{e:dmesh00}. 
Next, we saw \cref{e:dmesh0} already in \cref{e:firstn}. 
If $n$ and $n+1$ belong to $I_k$, then \cref{e:stayn+1} yields 
\cref{e:dmeshnew}. It remains to consider the case 
when $n\in I_k$ and $n+1\in I_{k+1}$. 
Then $n=i(k+1)-1$ and $n+1 = i(k+1)$; thus, 
\cref{e:Qk+1'} and \cref{e:Qk+1} yield
\begin{equation}
d_{n+1} = \frac{1}{Q_{k+1}} \leq 
\frac{d_{i(k+1)-1}}{1+64d^2_{i(k+1)-1}} = 
\frac{d_{n}}{1+64d^2_{n}},
\end{equation}
which is again \cref{e:dmeshnew}. 

Altogether, we've shown that the mesh satisfies 
\cref{e:dmesh00} and \cref{e:dmesh11}. 
In particular, the conclusion of \cref{t:main1} holds true. 

\subsection{Towards the Ces\`aro means}

Now consider the Ces\`aro means of $(\bx_n)_{n\geq 1}$:
\begin{equation}
\by_n := \frac{1}{n}\sum_{k=1}^n \bx_k = 
\frac{1}{n}\sum_{k=1}^n \bT^{k-1}\bx_1.
\end{equation}
We know that $(\by_n)_{n\geq 1}$ converges weakly to $\bzero$, 
and in the next two subsections, we obtain the last 
main result \textbf{R4} announced in \cref{s:intro}.

\subsection{A large subsequence of Ces\`aro means}

Consider the block $I_k = \{i(k),\ldots,i(k+1)-1\}$. 
We have $t_{i(k)}<t_{i(k)}+w_k$ and
$t_{i(k+1)}\geq t_{i(k)}+w_k = t_{i(k)} + k+1 \geq t_{i(k)}+2$. 
Hence the following is well-defined: 
\begin{equation}
j(k) := \min\menge{n\in I_k}{t_n \geq t_{i(k)}+1}
\;\;\text{and}\;\;
J_k := \{i(k),\ldots,j(k)\}\subseteq I_k \cap \{1,2,\ldots,j(k)\}. 
\end{equation}
The corresponding mesh values start with 
$t_{i(k)}$. By definition $t_{j(k)-1}<t_{i(k)}+1$. 
Hence 
\begin{equation}
\label{e:jackr2}
t_{j(k)} = t_{j(k)-1} + d_{j(k)-1} < t_{i(k)}+1 + 1/Q_k 
\leq t_{i(k)}+1+1/8 = t_{i(k)}+9/8. 
\end{equation}
Now let $i,j$ be in $J_k$. 
From \cref{e:jackr2}, 
\begin{equation}
\label{e:jackr11}
(\forall i,j\in J_k)\quad |t_i-t_j| \leq t_{j(k)} - t_{i(k)} < 9/8.
\end{equation}

We now turn to the size of $J_k$.
First, we have the lower bound 
\begin{equation}
\label{e:Jklow}
\# J_k \geq Q_k
\end{equation}
because each mesh increment is at most $1/Q_k$. 

If $t_n < t_{i(k)}+1$, then \cref{e:jackr0} yields 
$1/d_n = Q_k+64(t_n-t_{i(k)})<Q_k+64$ and so 
\begin{equation}
d_n>\frac{1}{Q_k+64}. 
\end{equation}
Hence after at most $Q_k+64$ steps, we must have advanced from $t_{i(k)}$ to a 
mesh point exceeding $t_{i(k)}+1$. 
So we get the upper bounds 
\begin{equation}
\label{e:Jkup}
\# J_k \leq Q_k+65
\;\;\text{and}\;\; j(k) \leq i(k)+Q_k+64.
\end{equation}
We deduce from \cref{e:Jklow} and \cref{e:Jkup} that
\begin{align}
\label{e:jackr10}
\frac{\# J_k}{j(k)} \geq \frac{Q_k}{i(k)+Q_k+64}
\geq \frac{Q_k}{2Q_k+64}. 
\end{align}

Combining \cref{e:jackr11}, \cref{e:jackr10}, and \cref{l:Cesaro}, 
we obtain 
\begin{equation}
\|\by_{j(k)}\| \geq \frac{Q_k}{2Q_k+64} \cdot \rho_\infty \exp(-(9/8)^2/2)
\end{equation}
and therefore (because $Q_k\to\infty$) 
\begin{equation}
\varlimsup_{n\to\infty} \|\by_n\| \geq 
\varlimsup_{k\to\infty} \|\by_{j(k)}\| \geq \frac{\rho_\infty}{2}\exp(-81/128)>0.
\end{equation}

\subsection{A small subsequence of Ces\`aro means}

In this subsection, 
we again consider the block $I_k = \{i(k),\ldots,i(k+1)-1\}$, 
but this time we set for convenience 
\begin{equation}
j(k) := \max I_k = i(k+1)-1 \;\;\text{so that}\;\; I_k = \{i(k),\ldots,j(k)\}.
\end{equation}
Note that 
\begin{equation}
\# I_k = j(k)-i(k)+1 = i(k+1)-i(k). %\geq w_k = k+1 \to\infty.
\end{equation}
We now consider the Ces\`aro means over the block $I_k$ alone:
\begin{equation}
\label{e:zesaro}
\bz_k := \frac{1}{\# I_k}\sum_{n\in I_k} \bx_n. 
\end{equation}
Because $i(k+1)$ is the first index in block $I_{k+1}$, 
it follows from the definition of leaving block $I_k$ 
(see \textbf{Case~2} above) that 
\begin{equation}
\sum_{n\in I_k} d_n 
=t_{j(k)+1} - t_{i(k)} = t_{i(k+1)} - t_{i(k)} \geq w_k. 
\end{equation}
Recalling that $d_n\leq 1/Q_k$ for $n\in I_k$ (see \cref{e:jackr1}), 
we conclude that 
$w_k \leq \sum_{n\in I_k} d_n \leq (\# I_k)/Q_k$ and thus 
\begin{equation}
\label{e:jackr12}
\# I_k \geq w_k Q_k \geq w_k i(k). 
\end{equation}
Next, we consider $t_n,t_{n+1}$, where $n,n+1$ are both in $I_k$. 
By \cref{e:jackr1}, we have 
\begin{equation}
d_n = t_{n+1}-t_n > \frac{1}{Q_k+64w_k} =: \alpha_k.
\end{equation}
Now $Q_k+64w_k \geq 8+64\cdot 2 = 136$, so $\alpha_k \leq 1/136<1$. 
It thus follows from \cref{l:esumineq} that 
\begin{equation}
\max_{i\in I_k} \sum_{j\in I_k}\exp\big(-(t_i-t_j)^2\big) \leq 
\frac{{1}+\sqrt{\pi}}{\alpha_k} = \big({1}+\sqrt{\pi}\big)\big({Q_k+64w_k}\big). 
\end{equation}
Using \cref{kurve}\cref{kurve1}, \cref{e:rho2}, \cref{e:jackr12}, 
and the facts that 
$w_k=k+1\to\infty$ and $Q_k \geq 2^{k+1}w_k\to \infty$, 
we obtain 
\begin{subequations}
\begin{align}
\|\bz_k\|^2
&=\frac{1}{(\# I_k)^2}\sum_{i,j\in I_k}\scal{\bx_i}{\bx_j}
=\frac{1}{(\# I_k)^2}\sum_{i,j\in I_k}\rho_i\rho_j\exp\big(-(t_i-t_j)^2\big)\\
&\leq\frac{1}{(\# I_k)^2}\sum_{i,j\in I_k}\exp\big(-(t_i-t_j)^2\big)
\leq\frac{1}{\# I_k} ({1}+\sqrt{\pi})(Q_k+64w_k) \\
&\leq \frac{1}{w_kQ_k} ({1}+\sqrt{\pi})(Q_k+64w_k)  = 
({1}+\sqrt{\pi})\bigg(\frac{1}{w_k} + \frac{64}{Q_k}\bigg)\\
&\to 0.
\end{align}
\end{subequations}
Hence 
\begin{equation}\label{e:zzero}
\|\bz_k\| \to 0.
\end{equation}
We now consider the (full) Ces\`aro mean
\begin{subequations}
\begin{align}
\by_{j(k)} 
&= \frac{1}{j(k)}\sum_{n=1}^{j(k)} \bx_n
= \frac{1}{j(k)}\sum_{n=1}^{i(k)-1} \bx_n + 
\frac{1}{j(k)}\sum_{n\in I_k} \bx_n
= \frac{1}{j(k)}\sum_{n=1}^{i(k)-1} \bx_n + 
\frac{\# I_k}{j(k)} \frac{1}{\# I_k}\sum_{n\in I_k} \bx_n\\
&= \frac{1}{j(k)}\sum_{n=1}^{i(k)-1} \bx_n + \frac{\# I_k}{j(k)}\bz_k.
\end{align}
\end{subequations}
Using $\|\bx_n\|\leq 1$ for all $n$, \cref{e:zesaro}, the triangle inequality, 
\cref{e:jackr12},
and \cref{e:zzero}, 
we have 
\begin{subequations}
\begin{align}
\|\by_{j(k)}\| 
&\leq 
\frac{i(k)-1}{j(k)} + \frac{\# I_k}{j(k)}\|\bz_k\|
\leq 
\frac{i(k)-1}{\# I_k} + \frac{j(k)-i(k)+1}{j(k)}\|\bz_k\|\\
&< \frac{i(k)}{\# I_k} + \|\bz_k\| \leq \frac{1}{w_k} + \|\bz_k\|\\
&\to 0.
\end{align}
\end{subequations}
Therefore, 
\begin{equation}
0\leq \varliminf_{n\to\infty} \|\by_n\| \leq 
\varliminf_{k\to\infty} \|\by_{j(k)}\| = 0.
\end{equation}

\subsection{Summary}

\label{ss:sumbiz}

To sum up, we have presented an example of a firmly nonexpansive mapping 
$\bT \colon \bX\to \bX$ with $\Fix \bT=\{\bzero\}$
such that for some $\bx_1\in \bX$, the sequence 
of Ces\`aro means defined by 
\begin{equation}
\by_n = \frac{1}{n}\sum_{k=1}^n \bT^{k-1}\bx_1
\end{equation}
satisfies 
\begin{equation}
\by_n \weakly \bzero,\;\; \varliminf_{n\to\infty} \|\by_n\| = 0,\;\; 
\text{ and } \;\;\varlimsup_{n\to\infty} \|\by_n\| > 0.
\end{equation}

\subsection*{Acknowledgments}
The authors thank the editors and the two anonymous 
reviewers for their constructive and helpful comments. 
The authors acknowledge the usage of \texttt{ChatGPT~5.5} leading eventually to  
the realization of the explicit curve $\bu$ in \cref{ss:L2} and to 
the mesh in \cref{s:bizarre}. 
The research of HHB was partially supported by a Discovery Grant
of the Natural Sciences and Engineering Research Council of
Canada. 

\noindent
\textbf{Author contributions}
Both authors contributed equally to this manuscript.

\noindent
\textbf{Funding}
HHB is partially supported by a Discovery Grant of the Natural Sciences 
and Engineering Research Council of Canada.

\noindent
\textbf{Data Availability Statement}
No datasets were generated or analyzed during the current study.

\subsection*{Declarations}

\noindent
\textbf{Ethics approval and consent to participate}
Not applicable.

\noindent 
\textbf{Conflict of interest}
The authors declare no conflict of interest.


\begin{thebibliography}{999}
\seppthree
%\small


\bibitem{Baillon75}
J.B.\ Baillon: Un th\'eor\`eme de type ergodic pour les 
contractions non lin\'eaires dans un espace de Hilbert, 
\emph{Comptes Rendus de l'Acad\'emie des Sciences. 
S\'erie~I. Math\'ematique}~280 (1975), 1511--1514.

\bibitem{Baillon76}
J.B.\ Baillon: Quelques propri\'et\'es de convergence 
asymptotique pour les contractions impaires, 
\emph{Comptes Rendus de l'Acad\'emie des Sciences. S\'erie~I. Math\'ematique}~283 (1976), 587--590. 

\bibitem{Bauschke07}
H.H.\ Bauschke: 
Fenchel duality, Fitzpatrick functions and the extension of 
firmly nonexpansive mappings, 
\emph{Proceedings of the American Mathematical Society}~135 (2007), 135--139. 
\url{https://doi.org/10.1090/S0002-9939-06-08770-3}


\bibitem{BC2017}
H.H.\ Bauschke and P.L.\ Combettes:
\emph{Convex Analysis and Monotone Operator Theory in Hilbert Spaces},
2nd edition, Springer, 2017.
\url{https://doi.org/10.1007/978-3-319-48311-5}

\bibitem{BMR}
H.H.\ Bauschke, E.\ Matou\v{s}kov\'a, and S.\ Reich:
Projection and proximal point methods: convergence results and counterexamples, \emph{Nonlinear Analysis}~56 (2004), 715--738. 
\url{https://doi.org/10.1016/j.na.2003.10.010}


\bibitem{BL}
Y.\ Benyamini and J.\ Lindenstrauss:
\emph{Geometric Nonlinear Functional Analysis (Volume~1)}, 
American Mathematical Society, 2000. 



\bibitem{DM}
L.\ Debnath and P.\ Mikusi\'nski:
\emph{Introduction to Hilbert Spaces with Applications}, 
3rd edition, Academic Press, 2005. 

\bibitem{GL}
A.\ Genel and J.\ Lindenstrauss:
An example concerning fixed points,
\emph{Israel Journal of Mathematics} 22 (1975), 81--86. 
\url{https://doi.org/10.1007/BF02757276}

\bibitem{Ghosh}
S.\ Ghosh:
The Basel Problem, arXiv manuscript, 2021.
\url{https://arxiv.org/abs/2010.03953}

\bibitem{Guler}
O.\ G\"uler: 
On the convergence of the proximal point algorithm for convex
minimization, \emph{SIAM Journal on Control and Optimization}~29 (1991), 403--419.
\url{https://doi.org/10.1137/0329022}

\bibitem{Hundal}
H.S.\ Hundal:
An alternating projection that does not converge in norm,
\emph{Nonlinear Analysis}~57 (2004), 35--61. 
\url{https://doi.org/10.1016/j.na.2003.11.004}

\bibitem{Knapp}
A.W.\ Knapp:
\emph{Basic Real Analysis},
digital second edition, Project Euclid, 2018.
\url{https://doi.org/10.3792/euclid/9781429799997}

\bibitem{KrengelLin}
U.\ Krengel and M.\ Lin: 
Order preserving nonexpansive operators in $L_1$, 
\emph{Israel Journal of Mathematics}~58 (1987), 170--192. 
\url{https://doi.org/10.1007/BF02785675}

\bibitem{Kreyszig}
E.\ Kreyszig:
\emph{Introductory Functional Analysis with Applications}, 
John Wiley \& Sons, 1989. 


\bibitem{Reich78}
S.\ Reich:
Almost convergence and nonlinear ergodic theorems,
\emph{Journal of Approximation Theory}~24 (1978), 269--272.
\url{https://doi.org/10.1016/0021-9045(78)90012-6}

\bibitem{RieszNagy43}
F.\ Riesz and B.\ Sz.-Nagy:
\"Uber Kontraktionen des Hilbertschen Raumes,
\emph{Acta Universitatis Szegediensis. Acta Scientiarum Mathematicarum}~10 (1943), 202--205.
\url{https://real.mtak.hu/213123/1/math_010_202-205.pdf} 

\bibitem{RieszNagy}
F.\ Riesz and B.\ Sz.-Nagy:
\emph{Functional Analysis},
Dover, 1990. 


\bibitem{SHS}
I.\ Steinwart, D.\ Hush, and C.\ Scovel:
An explicit description of the reproducing kernel Hilbert spaces of 
Gaussian RBF kernels,
\emph{IEEE Transactions on Information Theory}~52 (2006), 4635--4643.
\url{https://doi.org/10.1109/TIT.2006.881713}

\bibitem{vN}
J.\ von Neumann:
Proof of the quasi-ergodic hypothesis, 
\emph{Proceedings of the National Academy of Sciences}~18 (1932), 70--82.
\url{https://doi.org/10.1073/pnas.18.1.70}

\bibitem{Wittmann}
R.\ Wittmann:
Mean ergodic theorems for nonlinear operators,
\emph{Proceedings of the American Mathematical Society}~108 (1990),
781--788. 
\url{https://www.jstor.org/stable/2047801}

\end{thebibliography}
\end{document}